\documentclass[12pt]{amsart}

\author{Gary Kennedy}
\address{Department of Mathematics\\
  The Ohio State University\\
  1760 University Drive \\
  Mansfield, Ohio 44906}
\email[G.~Kennedy]{kennedy.28@osu.edu}
\author{Lee J. McEwan}
\address{Department of Mathematics\\
  The Ohio State University\\
  1760 University Drive \\
  Mansfield, Ohio 44906}
\email[L.~McEwan]{mcewan.1@osu.edu}

\usepackage{amsmath}
\usepackage{amssymb,amsmath}
\usepackage{amsthm}
\usepackage{bm}  
\usepackage{geometry} 
\usepackage[parfill]{parskip}    
\usepackage{graphicx}
\usepackage{amssymb}
\usepackage{epstopdf}
\usepackage{enumerate}
\usepackage{graphpap}
\usepackage{tikz}
\usepackage{enumitem}
\usepackage[all]{xy}

\usepackage{lipsum}

\usepackage{comment}

\usetikzlibrary{matrix,arrows,decorations.pathmorphing}
\newcommand{\HH}{\mathbf H}
\newcommand{\CC}{\mathbb C}
\newcommand{\VV}{\mathbf V}

\newcommand{\ibar}{\bm\tilde{\imath}}

\newcommand{\lcm}{\operatorname{lcm}}
\newcommand{\SL}{\operatorname{SL}}

\newcommand{\dt}{d_{\bullet}}
\newcommand{\ct}{c_{\bullet}}
\newcommand{\St}{S_{\bullet}}
\newcommand{\zetat}{\zeta_{\bullet}}
\newcommand{\HHt}{\HH_{\bullet}}
\newcommand{\VVt}{\VV_{\bullet}}
\newcommand{\chit}{\chi_{\bullet}}
\newcommand{\xit}{\xi_{\bullet}}

\newcommand{\llambda}{\boldsymbol{\lambda}}
\newcommand{\sw}{\operatorname{sw}}
\newcommand{\bs}{\boldsymbol{s}}

\newtheorem{theorem}{Theorem}
\newtheorem{corollary}[theorem]{Corollary}
\newtheorem{lemma}[theorem]{Lemma}

\newcommand{\qedwhite}{\hfill \ensuremath{\Box}}

\title{On the Milnor Fiber Boundary of a Quasi-Ordinary Surface}

\begin{document}

\begin{abstract}
We give a recursive formula, expressed in terms of the characteristic tuples, for the Betti numbers of the boundary of the Milnor fiber of an irreducible quasi-ordinary surface $S$
which is reduced in the sense of Lipman.
The singular locus of $S$ consists of two components, and for each 
component we introduce a sequence of increasingly simpler surfaces.
Our recursion depends on a detailed comparison of these two sequences.
In the penultimate section, we indicate how we expect pieces of these associated surfaces to glue
together to reconstruct the Milnor fiber of $S$ and its boundary.
\end{abstract}

\maketitle

 
\section{Introduction} \label{introduction}
 
It is well-known that a quasi-ordinary hypersurface germ (of any dimension) has many strong similarities to the germ of a classical plane curve (see especially \cite{Li1}, but also \cite{BMN}, \cite{Gau}, \cite{GPMN}, \cite{Li0}, \cite{Li2}, \cite{MN}). The Milnor fiber of an irreducible curve germ is constructed by plumbing its resolution graph, with the structure of the fiber corresponding to the rupture components. These in turn correspond to the essential terms of the Puiseux expansion of the curve. Something similar happens with the boundary of a quasi-ordinary surface $S$. A transverse slice by a hyperplane at a generic point of a singular component of $S$ cuts out a curve germ, which though typically not irreducible,  has the same Puiseux exponents in each component. The Milnor fiber of the slice 
can be regarded as the fiber of two fibrations, called \emph{horizontal} and \emph{vertical}.
We are concerned with the \emph{vertical fibration}, which is described in Section \ref{setup}.
As there are two components in the singular locus of $S$, there are two such spaces. Each space can be obtained by a relative version of the plumbing construction for a curve Milnor fiber: it is constructed by gluing standard pieces corresponding to  single terms of the Puiseux expansion.
 
For a general treatment of Milnor fiber boundaries for non-isolated surface singularities, see \cite{NS}.
 
Our starting point is a series defining the germ at the origin of an irreducible quasi-ordinary surface $S$. We assume this series contains only essential characteristic terms, each with coefficient 1:
\begin{equation} \label{prototype}
\zeta = \sum_{j=1}^{e} x_1^{\lambda_{1j}} x_2^{\lambda_{2j}}.
\end{equation}
(The terminology is explained in Section \ref{setup}.) Such a series was called a {\em prototype} in \cite{KM}. Working with just this case is not a serious restriction: see our remarks in Section \ref{setup}.  We also assume $\lambda_{11}\lambda_{12} \neq 0$; Lipman (\cite {Li2}) calls such a surface ``reduced". We note that this condition is preserved under the partial resolution procedure used in our basic construction.
In the literature $\zeta$ is often called a \emph{branch}.
The conjugates of $\zeta$ generate a function
\begin{equation} \label{fproduct}
f(x_1,x_2,y) = \prod(y - \zeta(x_1,x_2)) \qquad \text{(product over all conjugates)}
\end{equation}
whose vanishing locus is $S$.
This quasi-ordinary surface maps via the projection $\pi: (x_1,x_2,y) \to (x_1,x_2)$ onto $\CC^2$, and the map is unramified over the complement of the coordinate axes.
 
In \cite{KM} we presented a recursion for calculating some basic information about the quasi-ordinary surface $S$, using as input the same information for two associated surfaces, which we call its \emph{truncation} and its \emph{derived surface}.
In Section \ref{setup} we give the basic definitions, then recall our recursive setup and formulas.
Our recursion can be carried out in two different ways (depending on which of the two coordinate axes is sliced, as explained in Section \ref{setup}), and to draw our conclusions we need to make a precise comparison between the resulting sequences of surfaces and the bundles they carry. This elaborate but elementary comparison is carried out in Section \ref{comparing}.

The results of the comparison are applied in Section \ref{topresults}, which contains the main results of the article. Here we show (Theorem \ref{eigen}) that the vertical monodromies induce \emph{diagonalizable} actions on the part of homology corresponding to eigenvalue 1. This fact, combined with the Wang sequence for each vertical fibration and our vertical monodromy formula from \cite{KM}, allows us to calculate the first Betti number for the total space of each vertical fibration (Theorem \ref{firstbettivf}). The two spaces have the same first Betti number, and indeed the main result of Section \ref{comparing} implies that the two total spaces of vertical fibrations are ``nearly" homeomorphic (i.e., homeomorphic after removing some standard pieces)  \emph{as spaces} (not as fibrations). (This last result is not proved here, but will be important in the sequel.) The boundary of the Milnor fiber is the union of these two total spaces, glued along a torus. We thus deduce the first Betti number of this boundary (Theorem \ref{firstbetti}). In particular, the first Betti number of the boundary is always even.

We make some informal remarks in Section \ref{notes} further elucidating the topological nature of the boundary, and outlining results which we intend to present in detail in forthcoming work. In Section \ref{example}, we present an example of our calculations.


\section{Basic definitions and the recursive setup} \label{setup}

Let $f(x_1,x_2,y)$ be the defining function of $S$ in $\mathbb{C}^3$, as in (\ref{fproduct}). 
The \emph{Milnor fiber} $M$ is the intersection of the surface $f(x_1,x_2,y) = \epsilon$
with a small ball centered at the origin, where first the radius is chosen to be sufficiently
small and then $\epsilon$ is chosen to be sufficiently small.
We denote the boundary of $M$ by $\partial M$; it
is a closed 3-dimensional manifold (see \cite{Mi}, and \cite{D} for the non-isolated case).

In the following discussion we let $i=1 \text{ or } 2$. For convenience we also introduce a notational device: if $i=1$ then $\ibar=2$, and \emph{vice versa}.

We define the {\em Milnor fiber of a transverse slice of the $x_i$-axis} to be the set of points $(x_1,x_2,y)$ obtained by following, in order, three steps:
\begin{itemize}
 \item[(1)]  
  require that  $\|(x_1,x_2,y)\|\leq\delta$, a sufficiently small (positive real) radius,
 \item[(2)]  
  require  that $x_i$ be a fixed number sufficiently close to (but different from) zero,
\item[(3)]
  finally require that  $f(x_1,x_2,y)=\epsilon$, a number sufficiently close to (but different from) zero.
 \end{itemize}  
Denote this transverse Milnor fiber by $F_{x_i}$. By keeping $x_i$ fixed but letting $\epsilon$ vary over a circle centered at 0, we obtain the {\em horizontal fibration}.
Keeping $\epsilon$ fixed but letting $x_i$ vary over a circle centered at 0, we obtain the {\em vertical fibration}. 
Thus we have a fibration over a torus.
Let 
\begin{equation}
h_{q, x_i}:H_q(F_{x_i}) \to H_q(F_{x_i})
\end{equation}
and 
\begin{equation} \label{vmonodromy}
v_{q, x_i}:H_q(F_{x_i}) \to H_q(F_{x_i})
\end{equation}
be the respective monodromy operators, taking coefficients over $\mathbb{C}$;
we call them the \emph{horizontal monodromy} and \emph{vertical monodromy}.
Let $V_{x_i}$ be the total space of the vertical fibration; we will call it the {\em vertical fibration space}.
It is useful at times to work with an equivalent \emph{square ball} formulation:
\begin{equation}\label{vbundle}
V_{x_i} = \{f(x_1,x_2,y) = \epsilon \} \cap \{ |x_i| = \delta, |x_{\ibar}| \leq \delta \}
\end{equation}
(where $|\epsilon| \ll \delta$),
fibering over the circle $|x_i| = \delta$.

The graded characteristic functions
$$
\HH(i)(t)=\frac{\det(tI-h_{0, x_i})}{\det(tI-h_{1, x_i})}
\qquad
\text{and}
\qquad
\VV(i)(t)=\frac{\det(tI-v_{0, x_i})}{\det(tI-v_{1, x_i})}
$$
are called the {\em horizontal} and \emph{vertical monodromy zeta functions}.
 
As we mentioned in Section \ref{introduction},
we are working only with a \emph{prototype surface}
in whose defining series (\ref{prototype}) there are only essential terms, each taken with
coefficient 1.
To justify this restriction, we 
refer the reader to the discussion at the beginning of Section 2 of \cite{Li1}.
In a similar vein, Theorem 4.1 of our prior paper \cite{KM} says that a quasi-ordinary surface and its prototype have the same horizontal and vertical monodromy zeta functions.

In \cite{KM} we presented a recursion for calculating some basic information about the quasi-ordinary surface $S$, using as input the same information for two associated quasi-ordinary surfaces called its \emph{truncation} $\St$
and \emph{derived surface} $S'(i)$. The recursion leads to two sequences of associated surfaces,
as indicated in the following \emph{basic diagram}:

\begin{equation} \label{basicdiagram}
\xymatrix{
S \ar[d] \ar[r] &S'(i)\ar[d] \ar[r] &S''(i) \ar[d]\ar[r] &S^{(3)}(i) \ar[d]\ar[r] 
&\cdots \ar[r] &S^{(k)}(i) \ar[d] \ar[r]
&\cdots \ar[r]  &S^{(e-1)}(i) \ar@{=}[d] \\
\St &\St'(i) &\St''(i) &\St^{(3)}(i) & &\St^{(k)}(i) & &\St^{(e-1)}(i)
}
\end{equation}
The surface $S^{(k)}(i)$ is again a prototype, and it has $e-k$ essential terms,
while
each of the surfaces in the bottom row has a single essential term.
The recursion treats the exponents of (\ref{prototype}) in an asymmetric way;
thus there are two ways to carry out this recursion (depending on whether $i=1 \text{ or } 2$), and two possible sequences of associated surfaces.

The {\em truncation} of $S$ is the surface $\St$ determined by
\begin{equation}\label{truncation}
\zetat =  x_1^{\lambda_{11}} x_2^{\lambda_{21}},
\end{equation}
i.e., by the initial term of $\zeta$.
Each vertical arrow in the basic diagram (\ref{basicdiagram}) represents truncation.

We write the two exponents of $\zetat$ in lowest terms:
\begin{equation} \label{eq:standard1} 
\lambda_{i1} = \frac{n_{i}}{m_{i}},
\quad \text{where }
\gcd(n_{i}, m_{i}) = 1. 
\end{equation} 

As we showed in  \cite{KM}, the degree of $\St$ (i.e., the number of sheets for the projection $\pi$)
is given by 
$$
\dt=\lcm(m_{1},m_{2}).
$$
The quantity
\begin{equation} \label{defb}
b_{i}=\dt/m_{\ibar}
\end{equation}
occurs repeatedly in the calculations in this paper, and it has a geometric significance:
although the surface $\St$ is irreducible, its transverse slice $x_i=\text{constant}$ 
consists of $b_{i}$ irreducible curves, for each of which the number of sheets
(under projection to the line $x_i=\text{constant}$) is $m_{\ibar}$.
Another quantity which appears is
\begin{equation} \label{cdef}
\ct =\gcd(n_{1},n_{2}).
\end{equation}
We would like to know a simple geometric interpretation of this number,
which (along with $\dt$) is one of the two basic ingredients of our formula
in Theorem \ref{eigencalc}.

We define  
$r_{i}$ and $s_{i}$ to be the smallest nonnegative integers so that
\begin{equation} \label{eq:toric}
\begin{bmatrix} m_{\ibar} & n_{\ibar} \\ r_{i} & s_{i} \end{bmatrix}
\end{equation} 
has determinant 1.
The {\em derived surface} $S'(i)$ is the quasi-ordinary surface determined by the series
\begin{equation}
\zeta'(i)=\sum_{j=1}^{e-1}x_1^{\lambda'_{1j}(i)} x_2^{\lambda'_{2j}(i)},
\end{equation} 
where the new exponents are computed by these formulas:
\begin{equation} \label{eq:basic}
\begin{aligned}
\lambda'_{\ibar j} (i) &
= m_{\ibar}(\lambda_{\ibar,j + 1} - \lambda_{\ibar 1} + \dt\lambda_{\ibar 1}), \\
\lambda'_{i j} (i) &
=b_i (\lambda_{i,j + 1} - \lambda_{i1} + \dt \lambda_{i1}) 
+ b_{i}r_{i} \lambda_{i1}\lambda'_{\ibar j} (i).
\end{aligned}
\end{equation}
We introduce the following shorthand: for a finite sequence 
$\boldsymbol{s} = \{ s_j \}$, let
\begin{align}
T(\bs) = \{s_{j+1} - s_1 + \dt s_1\}.
\end{align}
(Note that this reduces the length of the sequence by one.)
Then the formulas in (\ref{eq:basic})
say that 
\begin{equation} \label{eq:basicshort}
\begin{aligned}
\llambda'_{\ibar} (i) &
= m_{\ibar} \,T(\llambda_{\ibar}), \\
\llambda'_{i} (i) &
= b_{i} \,T(\llambda_{i}) + b_{i}r_{i} \lambda_{i1} \, \llambda'_{\ibar} (i).
\end{aligned}
\end{equation}

\par
As indicated in our basic diagram (\ref{basicdiagram}), the  derivation process may be iterated.
The surface $S^{(k)}(i)$ will be called the \emph{$k^{th}$ derived surface}.
Working inductively, suppose that
$S^{(k)}(i)$ has defining series
\begin{equation} 
\zeta^{(k)}(i) = \sum_{j=1}^{e-k} x_1^{\lambda^{(k)}_{1 j}(i)} x_2^{\lambda^{(k)}_{2 j}(i)}.
\end{equation}
Write its leading exponents in lowest terms:
\begin{equation} \label{eq:standardk} 
\lambda^{(k)}_{i1} (i) = \frac{n_{i}^{(k)}(i)}{m_{i}^{(k)}(i)}\,,
\qquad\quad
\lambda^{(k)}_{\ibar 1} (i) = \frac{n_{\ibar}^{(k)}(i)}{m_{\ibar}^{(k)}(i)}\,.
\end{equation} 
Then the degree of its truncation $\St^{(k)}(i)$ is
\begin{equation} \label{degtruncation}
\dt^{(k)}(i)=\lcm(m_{1}^{(k)}(i),m_{2}^{(k)}(i)).
\end{equation}
As before we define
\begin{equation} \label{defbk}
b^{(k)}_{i}=\dt^{(k)}(i)/m_{\ibar}^{(k)}(i)
\end{equation}
and
\begin{equation} \label{defck}
\ct^{(k)} (i)=\gcd(n_{1}^{(k)}(i),n_{2}^{(k)}(i)),
\end{equation}
and again we define  
$r_{i}^{(k)}$ and $s_{i}^{(k)}$ to be the smallest nonnegative integers so that
\begin{equation} \label{eq:torick}
\det
\begin{bmatrix} m_{\ibar}^{(k)}(i) & n_{\ibar}^{(k)}(i) \\ r_{i}^{(k)} & s_{i}^{(k)} \end{bmatrix}
=1.
\end{equation} 
Then $S^{(k+1)}(i)$ is the quasi-ordinary surface determined by the series
\begin{equation} \label{derivedk}
\zeta^{(k+1)}(i) = \sum_{j=1}^{e-k-1} x_1^{\lambda^{(k+1)}_{1 j}(i)} x_2^{\lambda^{(k+1)}_{2 j}(i)}
\end{equation}
where the new exponents are computed by these formulas:
\begin{equation} \label{eq:basick}
\begin{aligned}
\lambda^{(k+1)}_{\ibar j} (i) &
= m_{\ibar}^{(k)}(i)
\left(\lambda^{(k)}_{\ibar,j + 1}(i) - \lambda^{(k)}_{\ibar 1}(i) + \dt^{(k)}(i)\lambda^{(k)}_{\ibar 1}(i)\right), \\
\lambda^{(k+1)}_{i j} (i) &
=b^{(k)}_i \left(\lambda^{(k)}_{i,j + 1}(i) - \lambda^{(k)}_{i1}(i) + \dt^{(k)}(i) \lambda^{(k)}_{i1}(i)\right) 
+ b^{(k)}_{i}r^{(k)}_{i} \lambda^{(k)}_{i1}(i)\lambda^{(k+1)}_{\ibar j} (i).
\end{aligned}
\end{equation}
Again this is made more palatable by a shorthand: 
for a finite sequence 
$\boldsymbol{s} = \{ s_j \}$, let
\begin{align}
T^{(k)}(\bs) = \{s_{j+1} - s_1 + \dt^{(k)}(i) s_1\}.
\end{align}
Then the formulas in (\ref{eq:basick})
say that 
\begin{equation} \label{eq:basickshort}
\begin{aligned}
\llambda^{(k+1)}_{\ibar} (i) &
= m_{\ibar}^{(k)}(i) \,T^{(k)}(\llambda^{(k)}_{\ibar}(i)), \\
\llambda^{(k+1)}_{i} (i) &
= b_{i}^{(k)} \,T^{(k)}(\llambda^{(k)}_{i}(i)) 
+ b^{(k)}_{i}r^{(k)}_{i} \lambda^{(k)}_{i1}(i) \, \llambda^{(k+1)}_{\ibar} (i).
\end{aligned}
\end{equation}


For the surface $S$, let $d$ denote its degree. Let 
$\chi(i)$ denote the Euler characteristic of the
transverse Milnor fiber $F_{x_i}$, and as above let
$\HH(i)$, and $\VV(i)$ denote its horizontal and vertical monodromy zeta functions.
For the truncation $\St$,
let $\dt$, $\chit(i)$, $\HHt(i)$, and $\VVt(i)$ denote 
the same quantities. For the derived surface $S'(i)$,
denote these quantities by 
$d'(i)$, $\chi'(i)$, $\HH'(i)$, and $\VV'(i)$. Theorem 4.3 of \cite{KM} gives us the following formulas:
\begin{enumerate}[label=(\alph*)]
\item \label{degreerecursion}
$$
d=\dt d'(i)
$$
\item  \label{chibase}
$$
\chit(i)=m_{\ibar}b_{i}+n_{\ibar}b_{i}-m_{\ibar}n_{\ibar}(b_{i})^2
=\dt+\dt\lambda_{\ibar 1}-(\dt)^2\lambda_{\ibar 1}
$$
\item  \label{chirecursion}
$$
\chi(i)=d'(\chit(i)-b_{i})+b_{i}\chi'(i)=d'\chit(i)+b_{i}(\chi'(i)-d')
$$
\item  \label{horbase}
$$
\HHt(i)(t)=\frac {(t^{\dt}-1)(t^{n_{\ibar}b_{i}}-1)} {(t^{n_{\ibar}\dt}-1)^{b_{i}}}
$$
 \item  \label{horcursion}
$$
\HH(i)(t)=\frac {\HHt(t^{d'})(\HH'(t))^{b_{i}}} {(t^{d'}-1)^{b_{i}}}
$$
\item \label{vertbase}
$$
\VVt(i)(t)=\frac {(t-1)^{\dt}} {(t^{n_{\ibar}b_{i}/\ct}-1)^{\ct(\dt -1)}}
$$
\item  \label{vertrecursion}
$$
 \VV(i)(t)=\frac{(\VVt(t))^{d'}\VV'(t^{b_{i}})} {(t^{b_{i}}-1)^{d'}}
$$
\end{enumerate}
Note that \ref{degreerecursion} implies that $d'(1)=d'(2)$; in the subsequent
formulas this quantity is simply denoted by $d'$.
\par
In a subsequent section, we will apply formulas \ref{vertbase}
and \ref{vertrecursion} to the $k^{th}$ derived surface $S^{(k)}(i)$ rather than to $S$,
and we record the needed versions of those formulas here.
Letting $d^{(k)}(i)$ and $\VV^{(k)}(i)$ denote the degree
and the vertical monodromy zeta function of this surface, and letting 
$\dt^{(k)}(i)$ and
$\VVt^{(k)}(i)$
denote the same quantities for its truncation $\St^{(k)}(i)$,
we have these formulas:
\begin{equation} \label{vertbasek}
\VVt^{(k)}(i)(t)
=\frac {\left(t-1\right)^{\dt^{(k)}(i)}} 
{\left(t^{n_{\ibar}^{(k)}(i)b_{i}^{(k)}/\ct^{(k)}(i)}-1\right)
^{\ct^{(k)}(i)
\left(\dt^{(k)}(i) -1\right)}}
\qquad \text{(for $k=0$ to $e-1$)},
\end{equation}
\begin{equation} \label{vertrecursionk}
 \VV^{(k)}(i)(t)=\frac{\left(\VVt^{(k)}(i)(t)\right)^{d^{(k+1)}(i)}
 \VV^{(k+1)}(i)\left(t^{b_{i}^{(k)}}\right)}
{\left(t^{b_{i}^{(k)}}-1\right)^{d^{(k+1)}(i)}}
\qquad \text{(for $k=0$ to $e-2$)}.
\end{equation}


\section{Comparing the two recursions} \label{comparing}

In the previous section we presented a derivation process leading
to the basic diagram of surfaces in (\ref{basicdiagram}). There are two such diagrams,
depending on our choice of $i=1 \text{ or } 2$. This section is devoted to a detailed
comparison of the surfaces in the two diagrams.
\par
Recall that $d^{(k)}(i)$ and $\dt^{(k)}(i)$ denote, respectively, the degrees
of the surfaces $S^{(k)}(i)$ and $\St^{(k)}(i)$.

\begin{theorem} \label{equald}
The degrees of the surfaces  in the two basic diagrams are the same:
\begin{enumerate}
\item $d^{(k)}(1)=d^{(k)}(2)$ for each $k=1,\dots,e-1$.
\item $\dt^{(k)}(1)=\dt^{(k)}(2)$ for each $k=1,\dots,e-1$.
\end{enumerate}
\end{theorem}
\begin{proof}
We first observe that the formulas which compute
$\dt^{(k)}(i)$ depend only on the first $k$ terms of the
series (\ref{prototype}). Thus we obtain the same value
for $\dt^{(k)}(i)$ if we replace this series by the shortened series
\begin{equation}  \label{shortened}
\tilde{\zeta} = \sum_{j=1}^{k} x_1^{\lambda_{1j}} x_2^{\lambda_{2j}}.
\end{equation}
Letting $\tilde{S}$ denote the surface defined by this shortened series,
the three surfaces
$S_{\bullet}^{(k)}(i)$, $\tilde{S}_{\bullet}^{(k)}(i)$, and $\tilde{S}^{(k)}(i)$
are the same.
Thus if we know that $d^{(k)}(1)=d^{(k)}(2)$ for each prototype surface,
we infer that $\dt^{(k)}(1)=\dt^{(k)}(2)$ as well.
We now prove the first equality by induction on $k$.
The base case $d'(1)=d'(2)$ has already been observed.
The inductive step is accomplished by applying
formula \ref{degreerecursion} of the previous section to the surface
$S^{(k)}(i)$; this tells us that
$$
d^{(k)}(i)=\dt^{(k)}(i) \, d^{(k+1)}(i).
$$
\end{proof}
In view of the theorem, we can simplify notation by writing $\dt^{(k)}$ instead of
$\dt^{(k)}(1)$ or $\dt^{(k)}(2)$. By repeated application of \ref{degreerecursion},
we infer that the degree of our surface $S$ is given by
\begin{equation} \label{degreeS}
d=\dt \dt' \dt^{(2)} \cdots \dt^{(e-1)}.
\end{equation}

\begin{corollary} \label{twolcms}
\quad $\dt^{(k)}=\lcm(m_{1}^{(k)}(1),m_{2}^{(k)}(1))
=\lcm(m_{1}^{(k)}(2),m_{2}^{(k)}(2))$
\end{corollary}
\begin{proof}
Refer to formula (\ref{degtruncation}).
\end{proof}


Now consider the first formula of (\ref{eq:basic}):
\begin{equation} \label{transverse}
\lambda'_{\ibar j} (1) 
= m_{\ibar}(\lambda_{\ibar,j + 1} - \lambda_{\ibar 1} + \dt\lambda_{\ibar 1}).
\end{equation}
One can interpret it as a formula to be applied 
to the plane curve with Puiseux series
\begin{equation} \label{transverserec}
\sum_{j=1}^{e} x_{\ibar}^{\lambda_{\ibar j}},
\end{equation}
thus obtaining a derived curve with Puiseux series
\begin{equation}
\sum_{j=1}^{e-1} x_{\ibar}^{\lambda'_{\ibar j}(1)}.
\end{equation}
This derivation process for a curve was used in \cite{KM},
and we need to understand some aspects of it here.
To begin, observe that the curve defined by (\ref{transverserec})
is one of the irreducible components of the transverse
slice of $S$ by $x_i=1$;
the other components have the same essential exponents
but differing coefficients.
\par
Consider the denominator ${m_{\ibar}^{(k)}(i)}$
of $\lambda^{(k)}_{\ibar 1} (i)$,
an exponent appearing in the series (\ref{derivedk}) for the  
$k^{th}$ derived surface, equivalently a denominator
appearing in the $k^{th}$ derived curve of the transverse slice by $x_i=1$.
It can also be computed 
from series (\ref{prototype})
without going through the derivation process,
as the following theorem shows.
\begin{theorem} \label{interpretm}
The exponent $\lambda_{\ibar k}$ may be written as a fraction with
denominator 
\begin{equation} \label{denominator}
m_{\ibar}  \, m'_{\ibar}(i) \, m_{\ibar}^{(2)}(i) \cdots m_{\ibar}^{(k-1)}(i)
\end{equation}
and with numerator relatively prime to $m_{\ibar}^{(k-1)}(i)$.
\end{theorem}
\begin{proof}
Write these exponents as follows: 
\begin{equation}
(\lambda_{\ibar k}, \lambda_{\ibar 2}, \lambda_{\ibar 3}, \dots, \lambda_{\ibar e})
=
\left(\frac{\alpha_1}{\beta_1}, \frac{\alpha_2}{\beta_1 \beta_2}, \frac{\alpha_3}{\beta_1 \beta_2 \beta_3},
\dots,\frac{\alpha_e}{\beta_1 \beta_2 \dots \beta_e}\right),
\end{equation}
where each numerator $\alpha_i$ is relatively prime to $\beta_i$.
By definition we have $\beta_1=m_{\ibar}$.
Applying the top formula of (\ref{eq:basic}) and observing that
$m_{\ibar}\lambda_{\ibar 1}=\beta_1 \lambda_{\ibar 1}$ is an integer,
we see that
\begin{equation}
(\lambda'_{\ibar 1}(i),\lambda'_{\ibar 2}(i), \dots, \lambda'_{\ibar,e-1}(i))
=
\left(\frac{\alpha'_2}{\beta_2}, \frac{\alpha'_3}{\beta_2 \beta_3}, 
\dots,\frac{\alpha'_{e}}{\beta_2 \beta_3 \dots \beta_{e}}\right),
\end{equation}
where again each $\alpha'_i$ is relatively prime to $\beta_i$.
This shows that $\beta_2=m'_{\ibar}(i)$.
Similarly, next apply the top formula of (\ref{eq:basick}) with $k=1$,
observing that 
$m_{\ibar}^{(1)}(i)\lambda'_{\ibar 1}=\beta_2 \lambda'_{\ibar 1}$ is an integer,
to see that
\begin{equation}
(\lambda^{(2)}_{\ibar 1}(i),\lambda^{(2)}_{\ibar 2}(i), \dots, \lambda^{(2)}_{\ibar,e-2}(i))
=
\left(\frac{\alpha^{(2)}_3}{\beta_3}, \frac{\alpha^{(2)}_4}{\beta_3 \beta_4}, 
\dots,\frac{\alpha^{(2)}_{e}}{\beta_3 \beta_4 \dots \beta_{e}}\right),
\end{equation}
with each $\alpha^{(2)}_i$  relatively prime to $\beta_i$,
and deduce that $\beta_3=m^{(2)}_{\ibar}(i)$.
Continue in this fashion to conclude that in general $\beta_k=m^{(k-1)}_{\ibar}(i)$.
\end{proof}
\par
The quantity in (\ref{denominator}) has a geometric meaning.
Consider again the surface $\tilde{S}$ defined by the shortened series
(\ref{shortened}).
A component of its transverse slice has Puiseux series
\begin{equation}
\sum_{j=1}^{k} x_{\ibar}^{\lambda_{\ibar j}}
\end{equation}
and this curve has degree 
$
m_{\ibar}  \, m'_{\ibar}(i) \, m_{\ibar}^{(2)}(i) \cdots m_{\ibar}^{(k-1)}(i).
$
\begin{theorem} \label{divides}
The degree of $\tilde{S}$ is $d=\dt \dt' \dt^{(2)} \cdots \dt^{(k-1)}$,
and this quantity divides 
\begin{equation} \label{productdegree}
\left(m_{1}  \, m'_{1}(2) \, m_{1}^{(2)}(2) \cdots m_{1}^{(k-1)}(2)\right)
\cdot
\left(m_{2}  \, m'_{2}(1) \, m_{2}^{(2)}(1) \cdots m_{2}^{(k-1)}(1)\right).
\end{equation}
\end{theorem}
\begin{proof}
The first equation is simply (\ref{degreeS}) applied to $\tilde{S}$. The
other statement is (5.9.4) of \cite{Li2}. Here is the essence of Lipman's argument.
As explained in our introduction, one obtains the defining function for the
quasi-ordinary surface $\tilde{S}$ by taking
a product over all conjugates of the branch (\ref{shortened}).
One can obtain all such conjugates by independently taking all possible conjugates
for both transverse slice curves (\ref{transverserec}) and forming all
possible combinations. This will give an equation of the degree given in (\ref{productdegree}).
There is likely to be redundancy, however, since two different choices of conjugates
for the curves may lead to the same branch of the surface. Thus one obtains
a power of the defining function.
\end{proof}

To continue our comparison of the two possible diagrams in (\ref{basicdiagram}), we make an
elementary observation in linear algebra,
recording it as Lemma \ref{simple}.
Given a matrix
\begin{equation}
U =
\begin{bmatrix}
U_{11} & U_{12} \\ 
U_{21} & U_{22}
\end{bmatrix} 
\quad \text{with } U_{22} \neq 0,
\end{equation}
denote its determinant by $\Delta$ and let
\begin{equation} \label{defineswap}
\sw(U) =
\begin{bmatrix}
\Delta/U_{22} & U_{12}/U_{22}\\ 
-U_{21}/U_{22} & 1/U_{22} 
\end{bmatrix},
\end{equation}
a matrix with determinant $U_{11}/U_{22}$.
(We have chosen notation to suggest the word ``swap.")
Let $x(1)$, $x(2)$, $y(1)$, $y(2)$ be indeterminates.
\begin{lemma}  \label{simple} 
These two equations are equivalent:
\begin{equation}
\begin{bmatrix}
x(1) \\ 
y(2)
\end{bmatrix}
= U
\begin{bmatrix}
x(2) \\ 
y(1)
\end{bmatrix}
\quad \Leftrightarrow \quad
\begin{bmatrix}
x(1) \\ 
y(1)
\end{bmatrix}
= \sw(U)
\begin{bmatrix}
x(2) \\ 
y(2)
\end{bmatrix}.
\end{equation}
Furthermore $\det (\sw(U))=1$ if and only if $U_{11}=U_{22}$. \qedwhite
\end{lemma}


As before, we use shorthand notation for finite sequences, e.g.
\begin{equation}
\begin{aligned}
\llambda_{\ibar} &
= \{ \lambda_{\ibar j} \}_{j=1}^{e}, \\
\llambda'_{\ibar} (i) &
= \{ \lambda'_{\ibar j}(i) \}_{j=1}^{e-1}, \\
\llambda^{(k)}_{\ibar} (i) &
= \{ \lambda^{(k)}_{\ibar j}(i) \}_{j=1}^{e-k}, \\
&\text{etc.}
\end{aligned}
\end{equation}

\begin{theorem}\label{matrixeq} 
For each $k$ from 1 to $e-1$,
there exists an $\SL(2, \mathbb{Z})$ matrix $M^{(k)}$ such that
$$
\begin{bmatrix}
\llambda^{(k)}_{1}(1) \\ 
\llambda^{(k)}_{2}(1)
\end{bmatrix}
= M^{(k)}
\begin{bmatrix}
\llambda^{(k)}_{1}(2) \\ 
\llambda^{(k)}_{2}(2)
\end{bmatrix} .
$$
\end{theorem}
\begin{proof}

\par
Write out explicitly the equations in (\ref{eq:basicshort}), using first the choice $(i,\ibar)=(1,2)$ and then the choice $(i,\ibar)=(2,1)$:
\begin{equation}
\begin{aligned}
\llambda'_{2} (1) &
= m_{2} \,T(\llambda_{2}) \\
\llambda'_{1} (1) &
= b_{1} \,T(\llambda_{1}) + b_{1}r_{1} \lambda_{11} \, \llambda'_{2} (1) \\
\llambda'_{1} (2) &
= m_{1} \,T(\llambda_{1}) \\
\llambda'_{2} (2) &
= b_{2} \,T(\llambda_{2}) + b_{2}r_{2} \lambda_{21}\, \llambda'_{1} (2).
\end{aligned}
\end{equation}
Eliminating $T(\llambda_{1})$ and $T(\llambda_{2})$,
we find this system of two equations:
\begin{equation}
\begin{aligned}
\llambda'_{1} (1) &
=\frac{b_1}{m_1} \llambda'_{1} (2) + b_{1}r_{1} \lambda_{11}\llambda'_{2} (1) \\
\llambda'_{2} (2) &
=\frac{b_2}{m_2} \llambda'_{2} (1) + b_{2}r_{2} \lambda_{21}\llambda'_{1} (2) \\ 
\end{aligned}
\end{equation}
and write it as a matrix equation
\begin{equation}
\begin{bmatrix}
\llambda'_{1}(1) \\ 
\llambda'_{2}(2)
\end{bmatrix}
= U'
\begin{bmatrix}
\llambda'_{1}(2) \\ 
\llambda'_{2}(1)
\end{bmatrix}
\end{equation}
in which
\begin{equation} \label{udef}
U' =
\begin{bmatrix}
b_1/m_1 & b_{1}r_{1} \lambda_{11} \\ 
 b_{2}r_{2} \lambda_{21} & b_2/m_2 
\end{bmatrix} .
\end{equation}
Observe that $b_1/m_1=\dt/(m_1 m_2)=b_2/m_2$.
Thus by Lemma \ref{simple}, the matrix $M'=\sw(U')$ has determinant 1.
\par
To show that $M'$
is the required matrix for the statement
of the theorem when $k=1$,
we need to establish that its entries are integers.
The bottom right entry of $M'$ is the integer
$m_2/b_2=(m_1 m_2)/\dt$.
By (\ref{defb}), we see that $b_{1}m_{2}=b_{2}m_{1}$,
and together with (\ref{eq:standard1}) this implies that the top right entry
is the integer $n_{1}r_{1}$; similarly the bottom left entry is $-n_{2}r_{2}$.
Again using  (\ref{defb}) and (\ref{eq:standard1}),
we see that the determinant of $U'$
is
\begin{equation}
\begin{aligned}
\Delta 
& = \frac{b_1}{m_1} \cdot \frac{b_2}{m_2} 
- b_{1}b_{2}r_{1}r_{2}\lambda_{11}\lambda_{21} \\
& = (\dt)^{2} \left( 1- m_{1}m_{2}r_{1}r_{2}\lambda_{11}\lambda_{21}\right) \\
& = (\dt)^{2} \left( 1- n_{1}n_{2}r_{1}r_{2}\right). 
\end{aligned}
\end{equation}
Since $b_2$ divides $\dt$, the top diagonal entry $m_2 \Delta/b_2$ 
of $M'$ is an integer. 

\par
We have thus established the base case in an inductive proof
of the theorem. 
Hypothesize inductively that
we have found a matrix
\begin{equation}
U^{(k)}
=\begin{bmatrix}
U^{(k)}_{11} & U^{(k)}_{12} \\ 
 U^{(k)}_{21} & U^{(k)}_{22} 
\end{bmatrix}
 \end{equation}
 
for which
\begin{equation} \label{uequation}
\begin{bmatrix}
\llambda^{(k)}_{1}(1) \\ 
\llambda^{(k)}_{2}(2)
\end{bmatrix}
= U^{(k)}
\begin{bmatrix} 
\llambda^{(k)}_{1}(2) \\ 
\llambda^{(k)}_{2}(1)
\end{bmatrix},
\end{equation}
having identical diagonal entries, and such that 
\begin{equation} \label{mk}
M^{(k)}=\sw(U^{(k)})
\end{equation}
has integer entries.
To simplify notations further, we suppress all superscripts $(k)$, and we replace each superscript
$(k+1)$ by a plus sign.
Thus, for example, we write equation (\ref{uequation}) as
\begin{equation} \label{uequation2}
\begin{bmatrix} 
\llambda_{1}(1) \\ 
\llambda_{2}(2)
\end{bmatrix} 
= U
\begin{bmatrix}
\llambda_{1}(2) \\ 
\llambda_{2}(1)
\end{bmatrix} ,
\end{equation}
and we aim to prove that there is an appropriate matrix $U^+$ for which
\begin{equation} \label{uequation+}
\begin{bmatrix}
\llambda^+_{1}(1) \\ 
\llambda^+_{2}(2)
\end{bmatrix}
= U^+
\begin{bmatrix}
\llambda^+_{1}(2) \\ 
\llambda^+_{2}(1)
\end{bmatrix}.
\end{equation}

In the second equation of (\ref{eq:basickshort}), let $(i,\ibar)=(1,2)$;
with our simplified notation, this says that 
\begin{equation}
\llambda^+_{1} (1) 
= b_{1} \,T(\llambda_{1}(1)) 
+ b_{1}r_{1} \lambda_{11}(1) \, \llambda^+_{2} (1).
\end{equation}
(Note that $b_1$ and $r_1$ mean $b^{(k)}_{1}$
and $r^{(k)}_{1}$, etc.)

Apply the inductive hypothesis, and then use the first equation of (\ref{eq:basickshort})
twice:
\begin{equation}
\begin{aligned}
\llambda^+_{1} (1) 
& = 
b_{1} \,T
\left(
U_{11}
\llambda_{1}(2)
+
U_{12}
\llambda_{2}(1)
\right) 
+ 
b_{1}r_{1} \lambda_{11}(1) \, \llambda^+_{2} (1) \\
& =
b_{1} U_{11} T(\llambda_{1}(2))
+
b_{1} U_{12} T(\llambda_{2}(1))
+ b_{1}r_{1} \lambda_{11}(1) \, \llambda^+_{2} (1) \\
& =
\frac{b_{1}}{m_{1}(2)} \, U_{11}\llambda^+_{1} (2) 
+
\frac{b_{1}}{m_{2}(1)} \, U_{12} \llambda^+_{2} (1) 
+ 
b_{1}r_{1} \lambda_{11}(1) \, \llambda^+_{2} (1) \\
& =
\frac{b_{1}}{m_{1}(2)}  \, U_{11} \llambda^+_{1} (2) 
+
\left( \frac{b_{1}}{m_{2}(1)} \, U_{12}  
+ 
b_{1}r_{1} \lambda_{11}(1) \right) \llambda^+_{2} (1).
\end{aligned}
\end{equation}
An identical calculation using the opposite choice $(i,\ibar)=(2,1)$
establishes that
\begin{equation}
\llambda^+_{2} (2) 
 = 
\frac{b_{2}}{m_{2}(1)}  \, U_{22} \llambda^+_{2} (1) 
+
\left( \frac{b_{2}}{m_{1}(2)} \, U_{21}  
+ 
b_{2}r_{2} \lambda_{21}(2) \right) \llambda^+_{1} (2).
\end{equation}

These calculations show that $U^+$ exists and give a recursion for
its entries:
\begin{equation} \label{urecursion}
U^+
 =
\begin{bmatrix} 
U^+_{11} & U^+_{12} \\ 
U^+_{21} & U^+_{22} 
\end{bmatrix} 
 =
\begin{bmatrix} 
\frac{b_{1}}{m_{1}(2)}  \, U_{11} 
& 
\frac{b_{1}}{m_{2}(1)} \, U_{12}  
+ 
b_{1}r_{1} \lambda_{11}(1) \\ 
\frac{b_{2}}{m_{1}(2)} \, U_{21}  
+ 
b_{2}r_{2} \lambda_{21}(2) 
& 
\frac{b_{2}}{m_{2}(1)}  \, U_{22}
\end{bmatrix}.
\end{equation}

By the inductive hypothesis we know that $U_{11}=U_{22}$.
The second part of Theorem~\ref{equald} and equation (\ref{defbk}) imply that
\begin{equation} \label{bmswap}
b_{1}m_{2}(1)=b_{2}m_{1}(2);
\end{equation}
thus the diagonal entries of $U^+$ are equal.
In fact the recursive formula for the bottom right entry, together with 
(\ref{defbk}),
give us an explicit formula for these diagonal entries:
\begin{equation} \label{lipmandiv}
\begin{aligned}
U^{+}_{22} 
&=
\frac{b_2}{m_2} \cdot
\frac{b'_{2}}{m'_{2}(1)} \cdot
\frac{b_{2}^{(2)}}{m_{2}^{(2)}(1)} \cdots
\frac{b_{2}^{(k)}}{m_{2}^{(k)}(1)} \\
&=
\frac
{\dt^{(0)}(i) \cdot
\dt^{(1)}(i) \cdot
\dt^{(2)}(i) \cdots
\dt^{(k)}(i)}
{m_1 \cdot
m'_{1}(1) \cdot
m_{1}^{(2)}(1) \cdots
m_{1}^{(k)}(1) \cdot
m_2 \cdot
m'_{2}(1) \cdot
m_{2}^{(2)}(1) \cdots
m_{2}^{(k)}(1)} \, .
\end{aligned}
\end{equation}
\par
(In this display, we have momentarily restored all superscripts, but 
will continue to suppress them in what follows.)
Finally we must show that the entries of 
\begin{equation} \label{mplus}
M^+
 =\
\begin{bmatrix}
M^+_{11} & M^+_{12} \\ 
M^+_{21} & M^+_{22} 
\end{bmatrix}
=\sw(U^{+})
\end{equation}
are integers.

By (\ref{defineswap}), $M^+_{22}$ is the reciprocal
of the quantity in (\ref{lipmandiv}), and Theorem \ref{divides}
says that this is an integer.
To show that the other entries of $M^+$ are integers, we first
establish the following recursive formulas for its entries:
\begin{equation} \label{m11}
M^+_{11} = m_1(2) b_2 s_1 s_2 M_{11} 
   - \dt \left( r_1 s_2 \frac{n_2(2)}{m_2(2)} + r_2 s_1 \frac{n_1(1)}{m_1(1)} \right)
\end{equation}
\begin{equation} \label{m12}
M^+_{12} = s_1 m_1(2) M_{12} + r_1 n_1(2)
\end{equation}
\begin{equation} \label{m21}
M^+_{21} = s_2 m_2(1) M_{21} - r_2 n_2(1)
\end{equation}
\begin{equation} \label{m22}
M^+_{22} = \frac{m_{2}(1)}{b_{2}}  \, M_{22} = \frac{m_{1}(2)}{b_{1}}  \, M_{22}
\end{equation}
Formula (\ref{m22}) is immediate from the diagonal (and equal) entries of (\ref{urecursion}).

\par

To derive formula (\ref{m12}),
we first combine the leading component of the vector equation (\ref{uequation}) with formula (\ref{eq:standardk}) to obtain
\begin{equation} \label{whiteboard}
\lambda_{11}(1) = \frac{n_1(1)}{m_1(1)} = U_{11} \frac{n_1(2)}{m_1(2)} + U_{12} \frac{n_2(1)}{m_2(1)}\,.
\end{equation}


Here is our derivation of formula (\ref{m12}):
\begin{equation*}
\begin{alignedat}{2}
M^+_{12} = \frac{U^+_{12}}{U^+_{22}} &=
	\frac{\frac{b_{1}}{m_{2}(1)} \, U_{12}  
	+ 
	b_{1}r_{1} \lambda_{11}(1)}{U^+_{22}}
	\qquad &\text{by (\ref{urecursion})} \\
&=\frac{\frac{b_{1}}{m_{2}(1)} \, U_{12}  
	+ 
	b_{1}r_{1} 
	\left(U_{11}\frac{n_1(2)}{m_1(2)} + U_{12}\frac{n_2(1)}{m_2(1)}\right)}
	{U^+_{22}}
	\qquad &\text{by (\ref{whiteboard})} \\
&=\frac{\frac{b_{1}}{m_{2}(1)} \, U_{12}
	\left( 1+r_1 n_2(1) \right)
	+ 
	b_{1}r_{1} 
	U_{11}\frac{n_1(2)}{m_1(2)} }
	{U^+_{22}} \\
&=\frac{b_{1} s_{1} \, U_{12}}
	{U^+_{22}}
	+
	\frac{
	r_{1} n_1(2)
	\frac{b_{1}U_{11}}{m_1(2)} }
	{U^+_{22}}
	\qquad &\text{by (\ref{eq:torick})} \\	
&=\frac{b_{1} s_{1} \, U_{12}}
	{\frac{b_1}{m_1(2)}U_{22}}
	+
	r_{1} n_1(2)
	\qquad &\text{by (\ref{bmswap}), using that $U^+_{11}=U^+_{22}$} \\
&=s_1 m_1(2) M_{12}
	+
	r_{1} n_1(2) \\
\end{alignedat}
\end{equation*}
A completely parallel derivation yields formula (\ref{m21}).
\par
The derivation of formula (\ref{m22}) is more elaborate.
Formula (\ref{whiteboard}) yields
\begin{equation} \label{part}
\frac{m_2(1) m_1(2) n_1(1)}{m_1(1)} \, M_{22} 
= m_2(1) n_1(2) + M_{12} \, m_1(2) n_2(1)
\end{equation}
and a completely parallel process yields its counterpart
\begin{equation} \label{counterpart}
\frac{m_1(2) m_2(1) n_2(2)}{m_2(2)} \, M_{22} 
= - M_{21} \, m_2(1) n_1(2) + m_1(2) n_2(1).
\end{equation}
We next observe, using (\ref{eq:torick}) twice, that
\begin{equation*}
1= \det
\begin{bmatrix} m_{1}(2) & n_{1}(2) \\ r_{2} & s_{2} \end{bmatrix}
\cdot
\det
\begin{bmatrix} m_{2}(1) & n_{2}(1) \\ r_{1} & s_{1} \end{bmatrix}.
\end{equation*}
Thus, using first formulas (\ref{m12}) and (\ref{m21}), and then 
formulas(\ref{part}) and (\ref{counterpart}),
\begin{equation*}
\begin{aligned}
M^+_{12} & M^+_{21} + 1 \\
&=
\left( s_1 m_1(2) M_{12} + r_1 n_1(2) \right)
\left( s_2 m_2(1) M_{21} - r_2 n_2(1) \right) \\
& \quad +
\left( s_2 m_1(2) - r_2 n_1(2) \right)
\left( s_1 m_2(1) - r_1 n_2(1) \right) \\
&=
\, m_1(2) m_2(1) s_1 s_2 \left( M_{12}M_{21}+1 \right) \\
& \quad - r_1 s_2   \left( -n_1(2) m_2(1) M_{21} + m_1(2)n_2(1)  \right) 
- r_2 s_1   \left( n_2(1)m_1(2) M_{12} + m_2(1)n_1(2)  \right) \\
&=
\, m_1(2) m_2(1) s_1 s_2 \left( M_{12}M_{21}+1 \right)  - m_1(2) m_2(1) M_{22} \left( r_1 s_2 \frac{n_2(2)}{m_2(2)} + r_2 s_1 \frac{n_1(1)}{m_1(1)}  \right).
\end{aligned}
\end{equation*}
Thus
\begin{equation*}
\begin{aligned}
M^{+}_{11} &= \frac{M^+_{12} M^+_{21} + 1}{M^{+}_{22}} \\
&= 
\frac{m_1(2) m_2(1) s_1 s_2 \left( M_{12}M_{21}+1 \right)}
{\frac{m_2(1)}{b_2}M_{22}} \\
& \qquad -
\frac{m_1(2) m_2(1) M_{22} \left( r_1 s_2 \frac{n_2(2)}{m_2(2)} + r_2 s_1 \frac{n_1(1)}{m_1(1)}  \right)}
{\frac{m_1(2)}{b_1}M_{22}} \\
&= 
m_1(2) b_2 s_1 s_2 M_{11} 
- m_2(1)b_1
\left(
r_1 s_2 \frac{n_2(2)}{m_2(2)} + r_2 s_1 \frac{n_1(1)}{m_1(1)}
\right) \\
&= 
m_1(2) b_2 s_1 s_2 M_{11} 
- \dt
\left(
r_1 s_2 \frac{n_2(2)}{m_2(2)} + r_2 s_1 \frac{n_1(1)}{m_1(1)}
\right). \\
\end{aligned}
\end{equation*}
\par
We have already remarked that $M^{+}_{22}$ is an integer.
Formulas (\ref{m12}) and (\ref{m21}) 
show that $M^{+}_{12}$ and $M^{+}_{12}$ are integers, since all the quantities
appearing on the right are integers.
In view of Corollary \ref{twolcms}, formula (\ref{m11}) shows that $M^+_{11}$ is an integer.
\end{proof}

In Theorem \ref{equald}, we have proved that the quantities $\dt^{(k)}(i)$
attached to the surfaces $\St^{(k)}(i)$ of our basic diagram (\ref{basicdiagram}) are
the same, no matter
whether we let $i=1 \text{ or } 2$.
We will conclude this section by proving that the 
same thing is true for the quantities $\ct^{(k)}(i)$.
To do this, we need the following elementary number-theoretic lemma.

\begin{lemma} \label{equalnums}
Suppose the four rational numbers
$\frac{n_1(1)}{m_1(1)}$,
$\frac{n_2(1)}{m_2(1)}$,
$\frac{n_2(2)}{m_2(2)}$,
$\frac{n_1(2)}{m_1(2)}$
are related as follows:
\renewcommand\arraystretch{1.6}
\begin{equation*}
\begin{bmatrix}
\frac{n_1(1)}{m_1(1)} \\
\frac{n_2(1)}{m_2(1)}
\end{bmatrix}
=
M
\begin{bmatrix}
\frac{n_2(2)}{m_2(2)} \\
\frac{n_1(2)}{m_1(2)}
\end{bmatrix}
\end{equation*}
\renewcommand\arraystretch{1}
where $M \in \SL(2, \mathbb{Z})$.
\par
Suppose furthermore that
\begin{equation*}
\lcm(m_1(1),m_2(1)) = \lcm(m_2(2),m_1(2)).
\end{equation*}
Then
\begin{equation*}
\gcd(n_1(1),n_2(1)) = \gcd(n_2(2),n_1(2)).
\end{equation*}
\end{lemma}
\begin{proof}
We can multiply each column vector by the least common multiple of its
denominators, preserving the equation while
not changing the greatest common divisor of its numerator;
this reduces us to the case where all four numbers are integers.
In this case, one argues as follows: since the entries of $M$ are integers,
the $\gcd$ of the vector on the right divides both entries of the vector on the left,
and thus divides its $\gcd$. Since $M$ is in $\SL(2, \mathbb{Z})$, one can reverse
the argument, using $M^{-1}$.
\end{proof}

\begin{theorem} \label{equalc}
For each $k=1,\dots,e-1$, we have $\ct^{(k)}(1)=\ct^{(k)}(2)$ .
\end{theorem}
\begin{proof}
Apply Lemma \ref{equalnums} using the four rational numbers
$\frac{n_1^{(k)}(1)}{m_1^{(k)}(1)}$,
$\frac{n_2^{(k)}(1)}{m_2^{(k)}(1)}$,
$\frac{n_2^{(k)}(2)}{m_2^{(k)}(2)}$,
$\frac{n_1^{(k)}(2)}{m_1^{(k)}(2)}$.
Equation (\ref{degtruncation}) tells us that the
least common multiples of the denominator pairs are $\dt^{(k)}(1)$ and $\dt^{(k)}(2)$,
and Theorem \ref{equald} asserts that they are equal.
Equation (\ref{defck}) tells us that the greatest common divisors 
of the numerator pairs are
$\ct^{(k)}(1)$ and $\ct^{(k)}(2)$, which are likewise equal. 
\end{proof}


\section{Topological results} \label{topresults}

We now embark on a study of the topology of the vertical fibration space $V_{x_i}$.
We find it convenient to use the square ball formulation as in display (\ref{vbundle}). Recall from (\ref{vmonodromy}) that $v_{q, x_i}$
denotes its monodromy in dimension $q$; we are going to analyze $v_{1, x_i}$.

We will write $v$ for the diffeomorphism of $F_{x_i}$ coming from the vertical fibration. It is shown in \cite{KM} that $F_{x_i}$ is naturally a union of pieces $F_r$
(for $r = 1, \dots, e$), with each piece arising from the $r^{th}$ stage of the resolution of the transverse slice. (These pieces are usually not connected.) The boundary between $F_r$ and $F_{r+1}$ is a set of cycles denoted by
$\mathcal{C}_r$, and otherwise the pairs of pieces are disjoint. There is also a set of boundary cycles $\mathcal{C}_e$ on the last piece $F_e$;
note that $\mathcal{C}_e$ is the boundary of $F_{x_i}$.
We show in \cite{KM} that $v$ preserves these pieces, permuting the elements
of each $\mathcal{C}_r$, and that the action on each $F_r$ is isotopic to a finite action. At the very end of the proof of Theorem 4.3 (8), we establish the following key fact.

\begin{theorem}\label{trans} 

The action of $v_{1, x_i}$ on the cycles in each $\mathcal {C}_r$ is transitive.  
\qedwhite

\end{theorem}

The next result will provide the basis for computing the Betti numbers of the Milnor fiber boundary from the vertical monodromies.

\begin{theorem}\label{eigen} 

The Jordan blocks of the 1-eigenspace of $v_{1, x_i}:H_1(F_{x_i}) \to H_1(F_{x_i})$ have size 1.

\end{theorem}

\begin{proof} 
There is a weight filtration on $H_1(F_{x_i}; \mathbb{C})$ defined by
\begin{align*}
W_0 &= \mbox{image of  } H_1(\sqcup_{r=1}^e \mathcal {C}_r) \text{ in } H_1(F_{x_i}),  \\
W_1 &= \mbox{image of  } H_1(\sqcup_{r=1}^e F_r) \text{ in } H_1(F_{x_i}), \\
W_2 &= H_1(F_{x_i}).
\end{align*}
The radical $R$ of the intersection pairing on $H_1(F_{x_i})$ is the image of $H_1(\partial F_{x_i})=H_1(\mathcal{C}_e)$  under inclusion, and thus is contained in $W_0$. The intersection pairing induces a perfect pairing on the quotients of the weight filtration, for which $W_0/R$ is dual to $W_2/W_1$ and for which $W_1/W_0$ is self-dual. This pairing is compatible with the action of $v_{1, x_i}$. Since $v$ acts finitely on each $F_r$, the restriction of $v_{1, x_i}$ acts finitely on $W_1$, hence is diagonalizable. On the other hand, by Theorem \ref{trans}, $v_{1, x_i}$ acts transitively on the cycles in $\mathcal {C}_r$, so the generator of the  fixed cycles inside the space generated by $\mathcal{C}_r$ is $\Sigma_s C_{r,s}$ (summing over all the elements of $\mathcal{C}_r$, with a consistent orientation).  This cycle is trivial in homology, since it bounds the union of components $F_i$, $i \leq r$. Therefore there are no non-trivial fixed cycles under the action of $v_{1, x_i}$ on $W_0$.  Because $v_{1, x_i}$ is diagonalizable on $W_0$, there is no 1-eigenspace on $W_0$ or on $W_0/R$. By duality, the 1-eigenspace of $H_1(F_{x_i})$ also does not meet $W_2/W_1$. Therefore the 1-eigenspace is contained in $W_1$, where $v_{1, x_i}$ is diagonalizable.
\end{proof} 

For more details about surface diffeomorphisms, see \cite{Wall} p.~269 ff., on which the previous argument is modeled.

\begin{corollary} \label{multiplicity}
Let $\xi(i)$ denote the dimension of the 1-eigenspace of $v_{1, x_i}$.
Then the multiplicity of $t-1$ in the monodromy zeta function $\VV(i)$
is $1-\xi(i)$.
\end{corollary}
\begin{proof}
This follows from Theorem \ref{eigen} and the fact that the vertical fibration space is connected.
\end{proof}

\par
In the next three theorems we will deduce a formula for the first Betti number of the vertical fibration space. Recall from formulas (\ref{degtruncation}) and (\ref{defck}) the quantities 
$\dt^{(k)}(i)$ and $\ct^{(k)}(i)$.
For the surfaces $S^{(k)}(i)$ and $\St^{(k)}(i)$, we denote by $\xi^{(k)}(i)$ 
and $\xit^{(k)}(i)$ the respective dimensions of the 1-eigenspaces of vertical monodromy.

\begin{theorem}\label{eigencalc}
For $k=0$ through $e-2$, the dimension of the 1-eigenspace of the vertical monodromy of $S^{(k)}(i)$ is given by
\begin{equation*}
\xi^{(k)}(i) = d^{(k+1)}(i)(\ct^{(k)}(i) - 1)(\dt^{(k)}(i) -1) + \xi^{(k+1)}(i).
\end{equation*}
In particular (taking $k=0$)
\begin{equation*}
\xi(i) = d'(\ct - 1)(\dt -1) + \xi'(i).
\end{equation*}
At the other extreme,
\begin{equation*}
\xi^{(e-1)}(i) = (\ct^{(e-1)}(i)-1)( \dt^{(e-1)}(i) - 1).
\end{equation*}
\end{theorem}

\begin{proof} 
Theorem \ref{trans}, Theorem \ref{eigen}, and Corollary \ref{multiplicity} apply to all the quasi-ordinary surfaces of our basic diagram (\ref{basicdiagram}).
Thus the multiplicity of $t-1$
in the vertical monodromy zeta function $\VV^{(k)}(i)$
is $1-\xi^{(k)}(i)$; similarly the multiplicity of $t-1$
in $\VVt^{(k)}(i)$ is $1-\xit^{(k)}(i)$.
Formula (\ref{vertrecursionk}) tells us that
$$
1-\xi^{(k)}(i) = d^{(k+1)}(i)(1-\xit^{(k)}(i)) + 1-\xi^{(k+1)}(i) - d^{(k+1)}(i),
$$
while formula (\ref{vertbasek}) implies that
\begin{equation} \label{trunck}
1-\xit^{(k)}(i) = \dt^{(k)}(i) - \ct^{(k)}(i)( \dt^{(k)}(i) - 1).
\end{equation}
Together they imply the first statement of the theorem.
Formula (\ref{trunck}) may also be applied with $k=e-1$;
since $S^{(e-1)}(i)$ and $\St^{(e-1)}(i)$ are the same surface,
this gives the last statement in the theorem.
\end{proof}

We now relate these quantities for the two possible sequences in our basic diagram
(\ref{basicdiagram}).

\begin{theorem}
For each $k$ from 0 to $e-1$, we have $\xi^{(k)}(1) = \xi^{(k)}(2)$.
\end{theorem}

\begin{proof}
Theorems \ref{equald} and \ref{equalc} tell us that $d^{(k)}(1) = d^{(k)}(2)$,
$\dt^{(k)}(1) = \dt^{(k)}(2)$, and
$\ct^{(k)}(1) = \ct^{(k)}(2).$
Thus the last statement of Theorem \ref{eigencalc}
implies that $\xi^{(e-1)}(1) = \xi^{(e-1)}(2)$.
The first statement of  Theorem \ref{eigencalc} then tells us, recursively, that 
$\xi^{(k)}(1) = \xi^{(k)}(2)$ for $k=e-2, e-3, \dots, 0$.

\end{proof}

Now let $h_{1}(i) = \operatorname{dim}H_1(V_{x_i})$ be the first Betti number of the vertical fibration space.

\begin{theorem}\label{firstbettivf}
\quad $h_{1}(i) = \xi(i) + 1$. 
\end{theorem}

\begin{proof} Working with
homology groups over $\mathbb{C}$,
we consider a portion of the Wang sequence of the 
vertical fibration:

\begin{equation}\label{wang}
H_2(V_{x_i}) \to H_1(F_{x_i})  \xrightarrow{v_{1, x_i} - 1}  H_1(F_{x_i}) \to H_1(V_{x_i})  \to H_0(F_{x_i})  \xrightarrow{v_{0, x_i} - 1}  H_0(F_{x_i}).
\end{equation}

(For the Wang sequence, see Lemma 8.4 of \cite{Mi}.)
The map $v_{0,x_i} - 1$ is trivial and $H_0(F_x) \cong \CC$.
Furthermore the kernel of $v_{1,x_i} - 1$ is the 1-eigenspace.
Taking an alternating sum of ranks, we find that
$$\xi(i) - \mu + \mu - h_{1}(i) + 1 = 0.$$
\end{proof}

\begin{corollary}\label{equalBetti} The first Betti numbers of the vertical fibration spaces are equal:
$$h_{1}(1) = h_{1}(2).$$
\qedwhite
\end{corollary} 

Henceforth we simply write $\xi$ instead of $\xi(1)$ or $\xi(2)$.
Recall that $M$ denotes the Milnor fiber of our quasi-ordinary surface $S$
and that $\partial M$ denotes its boundary.

\begin{theorem}
The boundary of the Milnor fiber is the union of the two vertical fibration spaces along
a common torus:
\begin{equation}\label{eq:boundary} 
\partial M = V_{x_1} \cup_{T^2} V_{x_2}.
\end{equation} 
\end{theorem}

\begin{proof}
Using the square ball formulation of (\ref{vbundle}),
one obtains the boundary of $M$ by letting either  $|x_1| = \delta$ or
$|x_2| = \delta$. These pieces are the vertical fibration spaces.
Their intersection, which lies inside $|x_1| = |x_2| = \delta$, consists of
a finite number of tori. Theorem \ref{trans}, applied with $r=e$, implies that in fact there is just
one torus.
\end{proof}

\begin{theorem} \label{firstbetti}
The first Betti number of $\partial M$ is given by 
$$h_1(\partial M) = 2\xi.$$
\end{theorem} 

\begin{proof} 
Consider a portion of the Mayer-Vietoris sequence:
\begin{equation*}
H_1(T^2)  \xrightarrow{i}  H_1(V_{x_1}) \oplus H_1(V_{x_2}) \to H_1(\partial M) \xrightarrow{\partial}  H_0(T^2)
\to H_0(V_{x_1}) \oplus H_0(V_{x_2})
\end{equation*}

The last map is injective, and thus $\partial=0$.
If one composes the map
\begin{equation*}
T^2 \to V_{x_1} \times V_{x_2}
\end{equation*}
with the product of projections onto the bases of these fibrations,
\begin{equation*}
V_{x_1} \times V_{x_2} \to S^1 \times S^1,
\end{equation*}
one obtains a finite covering space map.
Thus the induced map on homology 
\begin{equation*}
H_1(T^2) \to H_1(S^1) \oplus H_1(S^1)
\end{equation*}
is an isomorphism; hence $i$ is injective.
Thus $h_1(\partial M) = h_1(1)+h_1(2) -2 = 2\xi$
by Theorem \ref{firstbettivf}.  
\end{proof}

\section{Notes} \label{notes}
Our comments here are somewhat informal; we will present careful proofs elsewhere.
The Milnor fiber of an irreducible quasi-ordinary surface germ, or at least a good model of it, is essentially a relative version of a singular plane curve. Both the Milnor fiber and its boundary can be viewed in terms of families, over products of circles, of curve Milnor fibers. Moreover, these families are constructed from parametrized versions of the plumbing construction for irreducible curves, in which pieces of the Milnor fiber are defined iteratively in terms of Puiseux exponents. The corresponding ingredients in our case are the characteristic exponents and the Milnor fibers of single-term truncations of the quasi-ordinary parametrization, obtained after partial resolution.
Although this paper deals exclusively with surfaces, we believe it contains all the ingredients needed
for analysis of any quasi-ordinary hypersurface.
\par
In this paper, our comparison between the two derived surfaces $S^{(k)}(1)$ and $S^{(k)}(2)$ has been numerical:
we have given a detailed description of how certain invariants of the two surfaces are related.
In ongoing work, however, we have discovered that the comparison goes deeper.
In what follows, we will describe the deeper relationship between the two derived surfaces, and explain the greater importance of the matrix $M^{(k)}$ of display (\ref{mk}).

We have seen that the two vertical fibration spaces have the same Betti numbers. Moreover, it turns out that the vertical fibration spaces are related in a simple topological way. Consider the vertical fibration spaces for the truncation $\St$. From each bundle we remove $\dt$ thickened circles (solid tori). We will show that these punctured vertical fibration spaces are isomorphic (as {\em spaces}, not as bundles). Then we will iterate the construction. The punctured vertical fibration spaces of the $k^{th}$ truncations 
$\St^{(k)}(i)$ are again isomorphic. 
Copies of these basic pieces are glued together along tori to form the total vertical fibration spaces, analogous to the classical construction of the Milnor fiber of a curve. We need $\dt^{(k)}$ copies of the $(k-1)^{st}$ construction to glue on to the $k^{th}$ truncation. Thus the two vertical fibration spaces are isomorphic, once $d=\dt \dt' \dt^{(2)} \cdots \dt^{(e-1)}$ disjoint solid tori are removed from each one. This fact, and the next theorem, will not be proved in this article, but play an important role in its sequel. Writing

\begin{equation*}
M^{(k)}
 =\
\begin{bmatrix}
M^{(k)}_{11} & M^{(k)}_{12} \\ 
M^{(k)}_{21} & M^{(k)}_{22}
\end{bmatrix},
\end{equation*}

we claim that the mapping $$\rho: (x_1, x_2, z) \rightarrow (x_1^{M^{(k)}_{11}}x_2^{M^{(k)}_{21}}, x_1^{M^{(k)}_{12}}x_2^{M^{(k)}_{22}}, z)$$ restricts to a homeomorphism between prototypes of
$S^{(k)}(1)$ and $S^{(k)}(2)$, and this map is 
invertible away from the set $\{x_1x_2 = 0\}$. Note that the projections of these spaces to $\mathbb{C}^2_{x_1,x_2}$ are unramified over the axes. Thus the locus of points where $\rho$ is not defined consists of $d^{(k)}$ topological 4-balls. Removing these balls
from each space, we thus obtain the following conclusion.

\begin{theorem} \label{nextpaper}
The punctured Milnor fibers of the surfaces $S^{(k)}(1)$ and $S^{(k)}(2)$ are homeomorphic.
\qed
\end{theorem} 

Looking at the boundaries of these Milnor fibers, we conclude that they 
are homeomorphic away from $d^{(k)}$ disjoint thickened circles.
Furthermore, the analysis that leads to a proof of Theorem \ref{nextpaper}
shows that, in the decomposition (\ref{eq:boundary})
of each of these Milnor fibers, the two vertical fibration spaces $V_{x_1}$ and $V_{x_1}$
are homeomorphic away from these thickened circles.

Proofs of these statements should pave the way for analyzing the Milnor fiber of a quasi-ordinary hypersurface of arbitrary dimension.


\section{An example} \label{example}

Consider the surface with branch
\begin{equation*}
\zeta = x_1^{2/7} x_2^{4/5} + x_1^{5/14} x_2 + x_1^{2} x_2^{19/10}.
\end{equation*}

Here $e=3$. Following the layout of our basic diagram (\ref{basicdiagram}), we display the branches of the associated surfaces when $i=1$:

\small
{
\begin{align*} 
\zeta &= x_1^{2/7} x_2^{4/5} + x_1^{5/14} x_2 + x_1^{2} x_2^{19/10} &
\zeta'(1) &= x_1^{705/2} x_2^{141} + x_1^{373} x_2^{291/2} &
\zeta''(1) &= x_1^{1451} x_2^{573/2}
 \\
\zetat &= x_1^{2/7} x_2^{4/5} &
\zetat'(1) &= x_1^{705/2} x_2^{141} &
\zetat''(1) &= x_1^{1451} x_2^{573/2}
\end{align*}
}

When $i=2$, the branches of the surfaces are as follows:
\small
{
\begin{align*} 
\zeta &= x_1^{2/7} x_2^{4/5} + x_1^{5/14} x_2 + x_1^{2} x_2^{19/10} &
\zeta'(2) &= x_1^{141/2} x_2^{987} + x_1^{82} x_2^{2259/2} &
\zeta''(2) &= x_1^{305} x_2^{606303/2}
 \\
\zetat &= x_1^{2/7} x_2^{4/5} &
\zetat'(2) &= x_1^{141/2} x_2^{987} &
\zetat''(2) &= x_1^{305} x_2^{606303/2}
\end{align*}
}

In both cases, the degrees of these surfaces are:
\begin{align*} 
d &= 140 &
d' &= 4 &
d'' &= 2
 \\
\dt &= 35 &
\dt' &= 2 &
\dt'' &= 2
\end{align*}
We also have:
\begin{align*} 
\ct &= 2 &
\ct' &= 141 &
\ct'' &= 1
\end{align*}

Observe that
\begin{equation*}
M^{(1)}
\begin{bmatrix}
141/2  \\ 
987 
\end{bmatrix}
= 
\begin{bmatrix}
-23  & 2 \\ 
-12 & 1
\end{bmatrix}
\begin{bmatrix}
141/2  \\ 
987 
\end{bmatrix}
=
\begin{bmatrix}
705/2  \\ 
141 
\end{bmatrix} 
\qquad\text{and}\qquad
M^{(1)}
\begin{bmatrix}
82  \\ 
2259/2 
\end{bmatrix}
= 
\begin{bmatrix}
373  \\ 
291/2 
\end{bmatrix}
\end{equation*}
in accordance with Theorem \ref{matrixeq}. 
Similarly, with the auxiliary matrix $U^{(1)}$
we have
\begin{equation*}
U^{(1)}
\begin{bmatrix}
141/2  \\ 
141 
\end{bmatrix}
= 
\begin{bmatrix}
1  & 2 \\ 
12 & 1
\end{bmatrix}
\begin{bmatrix}
141/2  \\ 
141 
\end{bmatrix}
=
\begin{bmatrix}
705/2  \\ 
987 
\end{bmatrix} 
\qquad\text{and}\qquad
U^{(1)}
\begin{bmatrix}
82  \\ 
 291/2
\end{bmatrix}
= 
\begin{bmatrix}
373  \\ 
2259/2 
\end{bmatrix}.
\end{equation*}

The second auxiliary matrix $U^{(2)}$, calculated from $U^{(1)}$ by the recursion of 
(\ref{urecursion}), is
\begin{equation*}
U^{(2)}
=
\begin{bmatrix}
\frac{2}{2} \cdot 1 
& \frac{2}{1} \cdot 2 + 2 \cdot 0 \cdot \frac{705}{2} 
\\
\frac{1}{2} \cdot 12 + 1 \cdot 1 \cdot 987 
& \frac{1}{1} \cdot 1
\end{bmatrix}
=
\begin{bmatrix}
1 
& 4 
\\
993 
& 1
\end{bmatrix}
\end{equation*}
and thus
\begin{equation*}
M^{(2)}
=
\sw(U^{(2)})
=
\begin{bmatrix}
-3971 
& 4 
\\
-993 
& 1
\end{bmatrix}.
\end{equation*}
Observe that
\begin{equation*}
M^{(2)}
\begin{bmatrix}
305  \\ 
606303/2 
\end{bmatrix}
=
\begin{bmatrix}
1451  \\ 
573/2 
\end{bmatrix}.
\end{equation*}

Using the formulas of Theorem \ref{eigencalc}, we find these
dimensions of the 1-eigenspaces of the vertical monodromy:
\begin{align*} 
\xi'' &= (\ct''-1) (\dt''-1) = (1-1)(2-1) = 0
 \\
\xi' &= d''(\ct'-1)(\dt'-1)+\xi''=2(141-1)(2-1)+0=280
\\
\xi &= d'(\ct-1)(\dt-1)+\xi'=4(2-1)(35-1)+280=416
\end{align*}
Thus each of the two vertical fibration spaces has first Betti number 417,
and the first Betti number of the boundary of the Milnor fiber is 832.


{\bf Note:} The second author thanks Patrick Popescu-Pampu and Andr{\'a}s N{\'e}methi for much good advice and insight, received during a sabbatical trip to their respective institutions in the Spring of 2012. Their suggestions proved fruitful, if a bit delayed! This work was partially supported by a grant from the Simons Foundation (\#318310 to Gary Kennedy).



\begin{thebibliography}{99}


\bibitem{BMN}  C. Ban, L. J. McEwan and A. N{\'e}methi, On the Milnor Fiber of a Quasi-Ordinary Singularity,  {\em Canadian Journal of Mathematics} {\bf 54}(1) (2002), 55-70. 

\bibitem{D} A. Dimca, Singularities and Topology of Hypersurfaces, Universitext, Springer Verlag, 1992.

\bibitem{Gau} Y.-N. Gau, Embedded topological classification of quasi-ordinary singularities, {\em Memoirs of the AMS} {\bf 388}, 1988.

\bibitem{GPMN} P. D. Gonz\'alez P\'erez, L. J. McEwan, A. N\'emethi, The zeta function of a quasi-ordinary singularity II, {\em Topics in Algebraic and Noncommutative Geometry, Contemporary Math.} {\bf 324}, AMS (2003).

\bibitem{KM} G. Kennedy, L. J. McEwan, Monodromy of plane curves and quasi-ordinary surfaces, {\em J. of Sings},  {\bf 1}, (2010), 146 - 168.

\bibitem{Li0} J. Lipman, Quasi-ordinary singularities of embedded surfaces, {\em Thesis}, Harvard University (1965).

\bibitem{Li1} J. Lipman, Quasi-ordinary singularities of surfaces in  ${\bf C}^3$, {\em Proc. of Symp. in Pure Math.} {\bf 40}, Part 2 (1983), 161--172.

\bibitem{Li2} J. Lipman, Topological invariants of quasi-ordinary singularities, {\em Memoirs of the AMS} {\bf 388}, 1988.

\bibitem{MN} L. J. McEwan, A. N\'emethi, The zeta-function of a quasi-ordinary singularity, {\em Compositio  Math.} {\bf 140} (2004), 667--682.

\bibitem{Mi} J. Milnor, Singular Points of Complex Hypersurfaces, {\em Annals of Mathematics Studies 61}, PUP, Princeton, NJ (1968). 

\bibitem{NS} A. N{\'e}methi, A. Szilard, Milnor Fiber Boundary of a Non-isolated Surface Singularity, {\em Springer Lecture Notes in Mathematics} \#2037, Springer-Verlag, (2012). 

\bibitem{Wall} C. T. C. Wall, Singular points of plane curves, {\em London Mathematical Society Student Texts} {\bf 63}, Cambridge University Press, 2004.


\end{thebibliography}
\end{document}